\documentclass[11pt]{amsart}
\usepackage{amsmath, amssymb, amsthm}
\usepackage{hyperref}
\usepackage{tikz}
\usepackage{subfigure}
\usepackage{ytableau}
\usepackage{array}
\usepackage{amsfonts}
\usepackage{mathtools}
\theoremstyle{plain}
\newtheorem{theorem}{Theorem}[section]
\newtheorem{lemma}[theorem]{Lemma}
\newtheorem{corollary}[theorem]{Corollary}
\newtheorem{proposition}[theorem]{Proposition}
\newtheorem{conjecture}[theorem]{Conjecture}
\newtheorem{definition}{Definition}

\newtheorem{remark}[]{Remark}
\newtheorem{claim}[]{Claim}

\renewcommand{\Im}{\operatorname{Im}}

\numberwithin{equation}{section}
\begin{document}
	\title[Zeros in the character table of the wreath product of the symmetric group
	]
	{Zeros in the character table of the wreath product of the symmetric group} 
	\author{JAYANTA BARMAN}
	\address{JAYANTA BARMAN\\ Department of Mathematics \\
		Indian Institute of Technology Kharagpur \\
		Kharagpur-721302,  India.} 
	\email{b1999jayanta@gmail.com, b1999jayanta@kgpian.iitkgp.ac.in}
	
	\author{Kamalakshya Mahatab}
	\address{Kamalakshya Mahatab\\ Department of Mathematics \\
		Indian Institute of Technology Kharagpur \\
		Kharagpur-721302,  India.} 
	\email{accessing.infinity@gmail.com, kamalakshya@maths.iitkgp.ac.in}
	
	\begin{abstract} 
		Let $G$ be a finite group with $t$ conjugacy classes, and let $S_N$ be the symmetric group. Let $Z_t(N)$ be the
		number of zeros in the character table of the wreath product
		$G\wr S_N$. We prove 
		\begin{equation*}
			Z_t(N)\ge \frac{2p_t(N)^{2}}{\log \frac{N}{t}}\left(1+O\left(\frac{\log \log \frac{N}{t}}{\log \frac{N}{t}}\right)\right),
		\end{equation*}
		where $p_t(N)$ is the number of $t$-multipartitions of $N$.
	\end{abstract}
	
	\maketitle
	\section{Introduction}\label{section1}
	Character values of finite groups play an important role in combinatorics, number theory, and representation theory. For various classes of finite groups, several methods have been developed to compute their character values. 
	While linear characters never vanish, a classical theorem of Burnside \cite{burnside1904arithmetical} states that every non-linear irreducible character of a finite group takes the value zero on some group element. Motivated by this phenomenon, considerable attention has been devoted to understanding the distribution of zeros and divisibility properties of character tables. For example, Peluse and Soundararajan \cite{peluse} proved that, for any fixed prime power, almost all entries in the character table of $S_N$ are divisible by that prime power. Let $Z(N)$ denote the number of zeros in the character table of $S_N$, and let $p(N)$ denote the number of partitions of $N$. Recently, Miller and Scheinerman \cite{miller} conducted a large-scale Monte Carlo simulation to determine the density of zeros in the character table of $S_N$ for large values of $N$, leading to the following conjecture:
	\begin{conjecture}
		$\frac{Z(N)}{p(N)^{2}} \sim  \frac{2}{\log N}$ as $N \to \infty$.
	\end{conjecture}
	In \cite{barman1}, we proved a lower bound matching the order in the above conjecture:
	\begin{align}\label{eq-zerosSN}
		Z(N) \ge \frac{2\,p(N)^2}{\log N} \left(1 + O\left(\frac{\log\log N}{\log N}\right)\right).
	\end{align}
	Our proof combines an explicit version of Tyler's theorem \cite{tyler} for $c(N,k)$ with a result of  Erdős and Lehner \cite{erdos} for $p(N,k)$, where $c(N,k)$ denotes the number of $k$-core partitions of $N$, and $p(N,k)$ denotes the number of partitions of $N$ in which no summand exceeds $k$.
	
	In this paper, we generalize the above lower bound to $G\wr S_N$, where $G$ is a finite group with $t$ conjugacy classes. For our proof, we require $c_t(N,k)$, the number of $t$-multipartitions of $N$ in which each component $\lambda^{(j)}$ is a $k$-core partition and $p_t(N)$, the number of $t$-multipartitions of $N$. We discuss these below.
	
	\subsection{Conjugacy classes of
		\texorpdfstring{$G\wr S_N$}{G wr S\_N}}
	Recall that the number of conjugacy classes of $S_N$ equals the number
	of its irreducible representations, which is precisely the number of
	partitions of $N$, denoted by $p(N)$. James and Kerber
	\cite[Corollary 4.4.4]{MR644144} showed that the number of irreducible
	representations of the wreath product $G\wr S_N$ is given by
	\begin{equation*}
		\sum_{\substack{N=n_1+n_2+\cdots+n_t\\ n_j\geq 0}}
		p(n_1)p(n_2)\cdots p(n_t).
	\end{equation*}
	This sum is precisely the number of $t$-multipartitions of $N$. There are also applications of $p_t(N)$ in the representation theory of Lie algebras \cite{bouwknegt,fayers}. In \cite{bouwknegt}, multipartitions play an important role in the study of Durfee systems and in establishing their existence.

	A $t$-multipartition of a positive integer $N$ is a sequence $\lambda = (\lambda^{(1)}, \lambda^{(2)}, 
	\dots, \lambda^{(t)})$, where each $\lambda^{(j)}$ is an integer partition (possibly empty) for all $j$ such that $\sum_{j=1}^{t}\left|\lambda^{(j)}\right|=N$. Here, $\left|\lambda^{(j)}\right|$ denotes the sum of the parts of the partition $\lambda^{(j)}$. Asymptotic formulas and structural
	properties of $p_t(N)$ have been extensively studied in
	\cite{andrews,atkin,barman3,murty}. Note that $p_t(N)$ coincides with the number of $t$-colored partitions
	of $N$. Various properties of $p_t(N)$,
	including log-concavity, arithmetic properties, and multiplicative
	properties, have also been studied in
	\cite{bringmann,bringmann1,chern}. In \cite{barman3}, we proved the following asymptotic formula for fixed
	$t$.
	\begin{theorem}[{\cite[Corollary 1.4]{barman3}}]\label{thm-murty}
		Let $t$ be any fixed
		positive integer. Then,
		\begin{equation*}
			p_{t}(N)=\left(\frac{t}{24}\right)^{\frac{t+1}{4}}\frac{\exp\left(\frac{2\pi}{\sqrt{6}}\sqrt{Nt}\right)}{\sqrt{2}N^{\frac{t+3}{4}}}\left(1+O\left(N^{-\frac{1}{2}}\right)\right).
		\end{equation*}  
	\end{theorem}
	In the following subsection, we discuss the generating function and asymptotic results for $c_t(N,k)$.
	\subsection{Asymptotic results for \texorpdfstring{$c_t(N,k)$}{}}
	Let $\lambda^{(1)}=(\lambda_1, \lambda_2,\ldots, \lambda_\ell)$ with $\lambda_1\ge \lambda_2\ge\cdots\ge \lambda_\ell\ge 0$ be a partition of $|\lambda^{(1)}|=\lambda_1+\lambda_2+\cdots+\lambda_\ell$. The \emph{Young diagram} associated with $\lambda^{(1)}$ is a left-justified array of unit boxes having $\lambda_j$ boxes in the $j$-th row, for $1\leq j\leq \ell$. For a box $b$ in the Young diagram of a partition $\lambda^{(1)}$, the hook associated with $b$ consists
	of $b$ together with all boxes directly to its right and below it. The hook length is the total number of boxes in the hook. 
	For example, in the Young diagram of $\lambda^{(1)} = (4,3,1)$ shown below, each box is labeled with its corresponding hook length.
	\begin{figure}[ht]
		\centering
		\ytableausetup{centertableaux}
		\begin{ytableau}
			6 &  4 & 3 & 1
			\\
			4 & 2 & 1 \\
			1\\
		\end{ytableau}
		
		\caption{Hook-lengths of $\lambda^{(1)}=(4,3,1)$. 
		}
	\end{figure}
	A partition is called a $k$-core if none of its hook lengths are divisible by $k$. For example, the partition $(4,3,1)$ is a $5$-core. Since $k$-core partitions form an important class of restricted partition functions in number theory and combinatorics, they have attracted considerable attention due to their rich arithmetic properties and numerous applications \cite{anderson, barman1,barman2,hanusa, kim}. In this paper, we generalize the notion of $k$-core partitions to $k$-core $t$-multipartitions and derive asymptotic formulas using the saddle-point method.
	
	Meinardus developed a general
	framework for this method
	\cite{meinardus1953asymptotische}, without relying on transformation
	properties of the generating function under the modular group. The
	method has subsequently been developed and applied to various
	partition functions
	\cite{murty}. Debruyne and Tenenbaum \cite{debruyne2020saddle} independently developed this method for general partition functions and derived several asymptotic formulas.
	Recently, Tyler \cite{tyler} obtained an asymptotic formula for the number of $k$-core partitions for arbitrary $k$ via the saddle-point method. In this paper, we solve the saddle-point equation for $c_t(N,k)$ explicitly for $y$ in terms of $N$, $k$, and $t$. Substituting this value into the general asymptotic formula yields an explicit asymptotic estimate for $c_t(N,k)$. Notably, one can obtain asymptotic formulas across any desired parameter range by explicitly solving the corresponding saddle-point equation. Since an asymptotic formula for $c_t(N,k)$ cannot be derived directly from that of $c(N,k)$, no such formula or nontrivial bounds for $c_t(N,k)$ appear to be known in the literature. To establish an asymptotic formula for $c_t(N,k)$, we first introduce the necessary notation and definitions.
	The generating function for $k$-core partitions $c(N,k)$ is given by
	\begin{align*}
		\sum_{N=0}^{\infty}c(N,k)q^{N} =\prod_{n=1}^{\infty}\frac{(1-q^{kn})^{k}}{(1-q^{n})}=q^{-\frac{k^{2}-1}{24}}\frac{\eta(kz)^{k}}{\eta(z)},      
	\end{align*}
	where $q=\exp(2\pi iz)$, $z=x+iy$ and $y>0$. The Dedekind eta function $\eta(z)$ is given by
	\begin{equation*}
		\eta(z)=\exp\left(\frac{\pi iz}{12}\right)\prod_{n=1}^{\infty}(1-\exp(2\pi inz)).  
	\end{equation*}
	We will use the functions $\mu_m$, $m\ge 1$, from \cite{tyler} throughout this article:
	\begin{equation}\label{eq-muk}
		\mu_{m}(z)=-\frac{z^{m+1}}{2\pi i} \left(\frac{d}{dz}\right)^{m} \log\eta(z).
	\end{equation}
	For $c_t(N,k)$, each component should be a $k$-core partition.
	Hence, the generating function for $c_t(N,k)$ is given by
	\begin{align*}
		F(q,t,k)=\sum_{N=0}^{\infty}c_t(N,k)q^{N} =\left(\prod_{n=1}^{\infty}\frac{(1-q^{kn})^{k}}{(1-q^{n})}\right)^{t}=q^{-\frac{(k^{2}-1)t}{24}}\frac{\eta(kz)^{kt}}{\eta(z)^{t}}.         
	\end{align*}
	By the Cauchy integral formula, we have
	\begin{equation*}
		c_{t}(N,k)=\frac{1}{2\pi i}\int_{\gamma}\frac{F(q,t,k)}{q^{N+1}}dq, 
	\end{equation*}
	where $\gamma$ is a simple positively oriented loop around the origin, located entirely in the unit circle. For a fixed value of $y$, as $x$ varies over any interval of length $1$, the variable $q$ moves along a full circle of radius $e^{-2\pi y}$. Therefore, we can write $c_t(N,k)$ as
	\begin{align}\label{eq-integral}
		c_{t}(N,k)&=\int_{-1/2}^{1/2} \exp\left( -2\pi izM\right)g_{t}(z,k)dx,
	\end{align}
	where 
	\begin{align}\label{eq-atz}
		g_{t}(z,k)=\frac{\eta(kz)^{kt}}{\eta(z)^{t}}\quad\text{and}\quad M=N+\frac{(k^{2}-1)t}{24}.
	\end{align}
	We state our general asymptotic formula for $c_t(N,k)$ in the following theorem.
	\begin{theorem}\label{thm-1.1}
		Let $t$ be a fixed positive integer, and let $k, N$ be arbitrary positive integers.
		\newline
		(i) Then there exists a unique solution $y>0$ such that
		\begin{align}\label{eq-t1.1}
			\frac{t(\mu_1(kiy)-\mu_1(iy))}{y^{2}}=M.
		\end{align}
		(ii) For this $y$, $c_t(N,k)$ satisfies
		\begin{align*}
			c_t(N,k)=\frac{y^{\frac{3}{2}}\exp(2\pi My)g_{t}(iy,k)}{\sqrt{t(\mu_{2}(iy)-\mu_2(kiy))}}\left(1+O\left(\frac{y}{t(\mu_{2}(iy)-\mu_2(kiy))}\right)\right).    
		\end{align*}
	\end{theorem} 
	We may extend the applications of the above formula by expressing $y$ explicitly in terms of $N$, $k$, and $t$.
	\begin{theorem}\label{thm-ctNm}
		Let $N$ be a large positive integer and $k\ge\frac{\sqrt{6}}{2\pi}\sqrt{\frac{N}{t}}\log \frac{N}{t}$. Then, 
		\begin{align*}
			c_t(N,k)=p_t(N)\exp\left(- kt\exp\left(-\frac{\pi k\sqrt{t}}{\sqrt{6N}}\right)\right)\left(1+O\left(N^{-\frac{1}{2}}(\log N)^{3}\right)\right).
		\end{align*}
	\end{theorem} 
	In the following subsection, we state our bound for the number of zeros in the character table of $G\wr S_N$ and provide an outline of the proof.
	\subsection{Zeros in the character table of \texorpdfstring{$G\wr S_N$}{G wr S\_N}}
	Let $Z_t(N)$ denote the number of zero entries in the character table of the wreath product $G \wr S_N$. We establish the following bound for $Z_t(N)$.
	\begin{theorem}\label{thm-ztNm}
		Let $t$ be any fixed positive integer. Then,
		\begin{align}\label{eq-Zero}
			Z_t(N)\ge \frac{2p_t(N)^{2}}{\log \frac{N}{t}}\left(1+O\left(\frac{\log \log \frac{N}{t}}{\log \frac{N}{t}}\right)\right).   
		\end{align}
	\end{theorem}
	The family of wreath products considered above contains several important infinite families of Weyl groups and complex reflection groups. In particular, the Weyl groups of type $A_N$ are the symmetric groups $S_N$.
	\begin{corollary}
		The Weyl groups of types $B_N$ and $C_N$ are isomorphic to the hyperoctahedral group $\mathbb{Z}/2\mathbb{Z}\wr S_N$. In the case of type $B_N$ (where $t=2$), and more generally for the complex reflection group $C_m\wr S_N$ (where $t=m$), equation \eqref{eq-Zero} holds.
	\end{corollary}
	
	We now give a sketch of the proof of the asymptotic formula for $c_t(N,k)$, which is a key part of this article.
	\subsection{\texorpdfstring{Sketch of the proof of the asymptotic result for $c_t(N,k)$}{}} 
	We use the integral in (\ref{eq-integral}) as the basis for our saddle-point analysis. 
	Note that $y=\Im z$ appears only on the right-hand side of (\ref{eq-integral}).
	Hence, we select $y$ appropriately as in Theorem~\ref{thm-1.1}$(i)$ to obtain an asymptotic formula for $c_t(N,k)$. 
	The Taylor series of $g_{t}(z,k)=\frac{\eta(kz)^{kt}}{\eta(z)^{t}}$ has radius of convergence $<y$, so we divide the integral in (\ref{eq-integral}) into two separate parts:
	\begin{align*}
		|x| < \frac{y}{3} 
		\qquad \text{and} \qquad 
		\frac{y}{3} \le |x| \le \frac{1}{2}.
	\end{align*}
	The integral in (\ref{eq-integral}) may be decomposed as
	\begin{equation}
		\begin{aligned}\label{eq-ptN1}
			c_{t}(N,k)&=\exp(2\pi My)g_{t}(iy,k)\int_{-y/3}^{y/3}\exp\left(-2\pi i Mx+2\pi i\frac{1}{2\pi i}\log\frac{g_{t}(z,k)}{g_{t}(iy,k)}\right)dx\\
			&+\exp(2\pi My)g_{t}(iy,k)\int_{\frac{y}{3}\le |x|\le \frac{1}{2}}\exp\left(-2\pi i Mx\right)\frac{g_{t}(z,k)}{g_{t}(iy,k)}dx.
		\end{aligned}
	\end{equation} 
	The Taylor expansion of $\log g_{t}(z,k)$ around $x=0$ is given by
	\begin{align}\label{eq-tyler1}
		\log g_{t}(z,k)&=\left(\log g_{t}(iy,k)+x\frac{d}{dz} \log g_{t}(iy,k)+\frac{x^{2}}{2!}\frac{d^{2}}{dz^{2}}\log g_{t}(iy,k)+\cdots\right).
	\end{align}
	From the definition of $\mu_m$, we have
	\begin{equation*}
		\left(\frac{d}{dz}\right)^{m}\log g_{t}(z,k)=t\frac{2\pi i}{z^{m+1}}(\mu_{m}(z)-\mu_m(kz)).
	\end{equation*}
	So, (\ref{eq-tyler1}) simplifies to
	\begin{align}\label{eq-taylor}
		\notag
		\frac{1}{2\pi i}\log \frac{g_{t}(z,k)}{g_{t}(iy,k)}
		&=t\left(x\frac{\mu_{1}(iy)-\mu_1(kiy)}{(iy)^{2}}+\frac{x^{2}}{2}\frac{\mu_{2}(iy)-\mu_2(kiy)}{(iy)^{3}}+\frac{x^{3}}{6}\frac{\mu_{3}(iy)-\mu_3(kiy)}{(iy)^{4}}\right)\\
		&+t\frac{x^{4}}{24}\frac{\mu_{4}(x^{\prime}+iy)-\mu_{4}(k(x^{\prime}+iy))}{(x^{\prime}+iy)^{5}},
	\end{align}
	for some $z^{\prime}=x^{\prime}+iy$, where $x^{\prime}$ is between $0$ and $x$. In the region $|x|<\frac{y}{3}$, we evaluate the integral using the Taylor expansion described above. However, this expansion is not valid in the region $\frac{y}{3} \le |x| \le \frac{1}{2}$. Instead, we use Proposition~3.1 of \cite{tyler}, which provides an upper bound. Combining these results, Proposition \ref{thm-main} shows that $|x|<\frac{y}{3}$ contributes to the main term, and $\frac{y}{3}\le |x|\le \frac{1}{2}$ contributes to the error in (\ref{eq-ptN1}).
	
	\section{Acknowledgments}
	The results in this paper were presented at the International Conference on Combinatorics 2026. We are grateful to the conference participants for their helpful comments, special for the suggestions of Digjoy Paul, Pooja Singla, and S.~Velmurugan.
	J. Barman is deeply thankful to the University Grants Commission (UGC), India, for their invaluable support through the Fellowship Programme. K. Mahatab is supported by the ARG-MATRICS Programme (grant no. ANRF/ARGM/2025/002540/MTR).
	
	\section{Murnaghan-Nakayama Rule}  
	The Murnaghan--Nakayama rule is a powerful combinatorial tool for computing irreducible character values of the symmetric group $S_N$. Both recursive and non-recursive versions of this rule are known. Moreover, the rule admits a natural generalization to wreath products, allowing one to compute irreducible character values of $G\wr S_N$.
	
	The irreducible characters of $G\wr S_N$ are indexed by $t$-multipartitions of $N$. If $\chi^\lambda_\mu$ denotes the value of the irreducible character corresponding to $\lambda$ on the conjugacy class indexed by $\mu$, then the generalized Murnaghan--Nakayama rule computes $\chi^\lambda_\mu$ by successively removing rim hooks from the Young diagrams associated with the components of $\lambda$.
	\begin{definition}[Rimhook]
		Let $\lambda=(\lambda^{(1)},\ldots,\lambda^{(t)})$ be a $t$-multipartition. 
		A \emph{rimhook} of $\lambda$ is a collection of $t$ adjacent boxes in the Young diagram of some $\lambda^{(i)}$ that lie along the outer boundary. After removing these boxes, no boxes remain directly to the south or to the east of them. In addition, none of the boxes in the rimhook has a southeast neighbour(the rimhook cannot contain a $2\times2$ block of boxes).
	\end{definition}  
	
	\begin{figure}[ht]
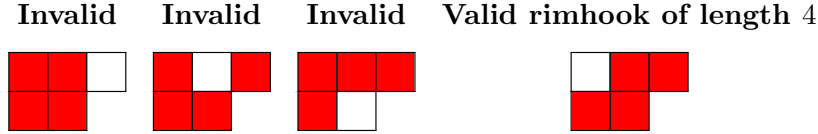

		\centering
		\ytableausetup{boxsize=1.3em}
		
		\begin{tabular}{cccc}
			
			\textbf{Invalid} & \textbf{Invalid} & \textbf{Invalid} & \textbf{Valid rimhook of length $4$} \\[2mm]
			
			\begin{ytableau}
				*(red) & *(red) & *(white) \\
				*(red) & *(red)
			\end{ytableau}
			&
			\begin{ytableau}
				*(red) & *(white) & *(red) \\
				*(red) & *(red)
			\end{ytableau}
			&
			\begin{ytableau}
				*(red) & *(red) & *(red) \\
				*(red) & *(white)
			\end{ytableau}
			&
			\begin{ytableau}
				*(white) & *(red) & *(red) \\
				*(red) & *(red)
			\end{ytableau}
			
		\end{tabular}
		
		\caption{Examples of three invalid and one valid rimhooks of length $4$ in the Young diagram of $\lambda^{(1)}=((3,2))$. 
		}
	\end{figure}

	\begin{definition}[Rimhook decomposition]
		Let $\lambda = (\lambda^{(1)}, \dots, \lambda^{(t)})$ and $\mu = (\mu^{(1)}, \dots, \mu^{(t)})$ be $t$-multipartitions of $N$. A \emph{rim-hook decomposition} of $\lambda$ with respect to $\mu$ is a sequence of successive rim-hook removals from the components of $\lambda$, where the lengths of the removed rim hooks correspond to the parts of $\mu$. The process terminates when all boxes of $\lambda$ are removed. We denote the set of all such decompositions by $\mathrm{RHD}(\lambda, \mu)$. Throughout the paper, we fix the order of rim-hook removals by arranging all parts of the multipartition in decreasing order. When two parts have the same size, the part belonging to the earlier component is removed first.
	\end{definition}
	
	\begin{definition}[Height]
		The \emph{height} of a rim hook is the number of rows it occupies minus one. For instance, the rim hook of length $4$ shown in Figure~2 has height $1$. If $\rho$ is a rim-hook decomposition, then the height of $\rho$, denoted by $ht(\rho)$, is defined as the sum of the heights of all rim hooks appearing in $\rho$.
	\end{definition}
	We now recall the generalized Murnaghan--Nakayama rule, which provides a combinatorial formula for computing irreducible character values of the wreath product $G\wr S_N$. For a more elementary form of this result, we refer the reader to \cite[Proposition~2.8]{dong}.
	\begin{proposition}[Non-recursive Murnaghan--Nakayama rule {\cite[Theorem 4.4.10]{MR644144}}]
		\label{proposition:MNrule}
		Let $\lambda$ and $\mu$ be $t$-multipartitions of $N$, and let $\chi^{1},\chi^{2},\ldots,\chi^{t}$ denote the irreducible characters of $G$ and $c_1, c_2,\dots,c_t$ denote the conjugacy classes of $G$.
		For each rimhook decomposition $\rho \in RHD(\lambda,\mu)$, define
		\begin{align*}
			\psi(\rho)&=\prod_{i=1}^{t}\prod_{j=1}^{t}\left( \chi^{i}(c_j)^{a_{ij}}\right),
		\end{align*}
		where $a_{ij}=\#\{\text{rimhooks $h\in \rho$ from  $\mu^{(j)}$ in $\lambda^{(i)}$ by $\rho$ }\}$.
		\newline
		Then,
		\begin{equation*}
			\chi^\lambda_\mu=\sum_{\rho\in RHD(\lambda,\mu)}(-1)^{ht(\rho)}\psi(\rho),
		\end{equation*}
		where $ht(\rho)$ represents the height of the rimhook decomposition $\rho$.
	\end{proposition}
	If $G$ is trivial, this reduces to the classical Murnaghan--Nakayama rule
	for the irreducible characters of $S_N$. 
	Thus, for $\lambda\vdash N$ and $\mu\in S_N$,
	\begin{equation*}
		\chi^\lambda_\mu=\sum_{\rho\in RHD(\lambda,\mu)}(-1)^{ht(\rho)}.
	\end{equation*}
	\textbf{Example.}
	Consider the $2$-multipartitions
	\[
	\lambda=((2,1,1,1),(2,2))
	\qquad\text{and}\qquad
	\mu=((3,2),(3,1)).
	\]
	\begin{displaymath}
		\ytableausetup{centertableaux}
		\ytableaushort
		{5 1, 3, 2, 1}
		* {2,1,1,1} 
		\hspace{0.2cm}
		\ytableausetup{centertableaux}
		\ytableaushort
		{3 2,2 1}
		* {2,2} 
		\longrightarrow
		\ytableausetup{centertableaux}
		\ytableaushort
		{\none\none, X, X, X}
		* {2,1,1,1} 
		\hspace{0.2cm}
		\ytableausetup{centertableaux}
		\ytableaushort
		{\none\none,\none\none}
		* {2,2} 
		\longrightarrow
		\ytableausetup{centertableaux}
		\ytableaushort
		{\none\none}
		* {2} 
		\hspace{0.2cm}
		\ytableausetup{centertableaux}
		\ytableaushort
		{\none X, X X}
		* {2,2} 
		\longrightarrow
		\ytableausetup{centertableaux}
		\ytableaushort
		{X X}
		* {2} 
		\hspace{0.2cm}
		\ytableausetup{centertableaux}
		\ytableaushort
		{\none }
		* {1} 
		\longrightarrow
		\ytableaushort
		{X}
		* {1} 
	\end{displaymath}
	The above figure illustrates one rim-hook decomposition of $\lambda$ by $\mu$. The first Young diagrams are labeled with their hook lengths. Since the largest part of $\mu$ is $3$, we begin by removing a rim hook of length $3$ from $\lambda$. As both $\lambda^{(1)}$ and $\lambda^{(2)}$ contain a rim hook of length $3$, this choice gives rise to two possible decompositions. In the figure, we choose the first possibility and indicate the removed boxes by ``X''. The resulting multipartition is
	\[
	((2),(2,2)).
	\]
	The next part to be removed is the part $3$ from $\mu^{(2)}$. At this stage there is only one possible rim hook of length $3$, whose removal yields
	\[
	((2),(1)).
	\]
	The remaining parts are then removed uniquely, and the process continues in the same manner until all boxes have been removed.
	
	Let $\rho_1$ denote the above rim-hook decomposition. Then
	\begin{align*}
		ht(\rho_1)=2+1+0+0=3, \quad\text{and} \quad a_{11}=2,\quad a_{12}=0,\quad a_{21}=0,\quad a_{22}=2.
	\end{align*}
	The second rim-hook decomposition of $\lambda$ by $\mu$ is as follows:
	\begin{displaymath}
		\ytableausetup{centertableaux}
		\ytableaushort
		{5 1, 3, 2, 1}
		* {2,1,1,1} 
		\hspace{0.2cm}
		\ytableausetup{centertableaux}
		\ytableaushort
		{3 2,2 1}
		* {2,2} 
		\longrightarrow
		\ytableausetup{centertableaux}
		\ytableaushort
		{\none\none, \none, \none, \none}
		* {2,1,1,1} 
		\hspace{0.2cm}
		\ytableausetup{centertableaux}
		\ytableaushort
		{\none X, X X}
		* {2,2} 
		\longrightarrow
		\ytableausetup{centertableaux}
		\ytableaushort
		{\none\none, X, X, X}
		* {2,1,1,1}
		\hspace{0.2cm}
		\ytableausetup{centertableaux}
		\ytableaushort
		{\none }
		* {1} 
		\longrightarrow
		\ytableausetup{centertableaux}
		\ytableaushort
		{X X}
		* {2} 
		\hspace{0.2cm}
		\ytableausetup{centertableaux}
		\ytableaushort
		{\none }
		* {1} 
		\longrightarrow
		\ytableaushort
		{X}
		* {1} 
	\end{displaymath}

	\begin{lemma}\label{lemma-Mur}
		Let $\lambda$ and $\mu$ be two $t$-multipartitions of $N$. If $\mu$ has a largest part of size $k$ and $\lambda$ is a $k$-core multipartion, then $\chi^{\lambda}_{\mu}=0.$   
	\end{lemma}
	
	Our proof of Theorem~\ref{thm-ztNm} relies on the following inequality (see Proposition~\ref{proposition:MNrule} and Lemma~\ref{lemma-Mur}), which follows from the generalized Murnaghan--Nakayama rule:
	\begin{align}\label{eq-main}
		Z_t(N)&\ge \sum_{k=1}^{N}c_t(N,k)\left(tp_{t,k}(N-k)-t^{2}p_{t,k}(N-2k)\right),
	\end{align}
	where $p_{t,k}(N-k)$ counts $t$-multipartitions of $N-k$ whose parts in every component 
	do not exceed $k$. Since $k$ can occur as the largest part in any of the $t$ components, we multiply by $t$. However, this may lead to overcounting, as adding a part $k$ to different components of a $t$-multipartition of $N-t$ can produce the same $t$-multipartition of $N$. For such an overcount to occur, two distinct components $t$-multipartitions of $N-t$ must have $k$ as a largest part, together with additional conditions ensuring that the resulting multipartitions coincide. Thus, the overcounted cases are bounded by $t^2p_{t,k}(N-2k)$.
	Applying Theorem~\ref{thm-ctNm}, Lemma~\ref{lemma-erdos}, and Lemma~\ref{lemma-erdos0} to the inequality \eqref{eq-main}, we obtain the lower bound for $Z_t(N)$.
	
	\section{Proof of Theorems~\ref{thm-1.1} and~\ref{thm-ctNm}}
	Recall that $M = N + \frac{(k^2-1)t}{24}$, and let $\vartheta \in \mathbb{C}$ satisfy $|\vartheta| \le 1$. Before proving Theorems~\ref{thm-1.1} and~\ref{thm-ctNm}, we establish an auxiliary proposition for the asymptotic formula of $c_t(N,k)$. Throughout the proof, the value of $\vartheta$ may depend on the parameters $N$, $k$, $t$, $x$, and $y$, and may vary from one occurrence to another. In Theorem~\ref{thm-1.1}, we prove that $y$ is the solution of 
	$ \frac{t\bigl(\mu_1(kiy) - \mu_1(iy)\bigr)}{y^2} - M = 0.$
	Hence, imposing this restriction on $y$ in the proposition does not restrict its generality.
	\begin{proposition}\label{thm-main} Let $y$ be chosen such that 
		\begin{align}\label{eq-assum}
			\left|\frac{t(\mu_{1}(kiy)-\mu_1(iy))}{y^{2}}-M\right|<\frac{2}{25y},
		\end{align} and assume that $\min(k,\frac{1}{y})\ge 1000$. Then,
		\begin{align*}
			c_t(N,k)=\frac{y^{\frac{3}{2}}\exp(2\pi My)g_{t}(iy,k)}{\sqrt{t(\mu_{2}(iy)-\mu_2(kiy))}}\left(1+\vartheta\frac{3.5y}{t(\mu_{2}(iy)-\mu_2(kiy))}\right).    
		\end{align*}
	\end{proposition}
	\begin{proof}
		Recall (\ref{eq-ptN1}), and apply Proposition~3.1 of \cite{tyler} to get
		\begin{align}\label{eq-p(N,t)}
			\notag
			c_t(N,k)&=\exp(2\pi My)g_{t}(iy,k)\int_{-y/3}^{y/3}\exp\left(-2\pi i Mx+2\pi i\frac{1}{2\pi i}\log\frac{g_{t}(z,k)}{g_{t}(iy,k)}\right)dx\\
			&+\exp(2\pi My)g_{t}(iy,k)\left(\vartheta y^{t} \exp\left(-\frac{t}{70}\min \left(k,\frac{1}{y}\right) \right)\right).
		\end{align}   
		From the Taylor expansion of $\frac{1}{2\pi i}\log\frac{g_t(z,k)}{g_t(iy,k)}$ in (\ref{eq-taylor}), we have
		\begin{align}\label{eq-tayfinal}
			\frac{1}{2\pi i}\log \frac{g_{t}(z,k)}{g_{t}(iy,k)}&=tx\frac{\mu_{1}(iy)-\mu_1(kiy)}{(iy)^{2}}
			+t\frac{x^{2}}{2!}\frac{\mu_{2}(iy)-\mu_2(kiy)}{(iy)^{3}}\left(1+2i\xi\frac{x}{y}+3\vartheta \frac{x^{2}}{y^{2}}\right),
		\end{align}
		where $\xi\in(-1,1)$ and $|\vartheta|\le 1$, due to Lemma~4.3$(iii)$ and Lemma~4.4 of \cite{tyler}.
		\newline
		Let
		\begin{align}\label{eq-alpha}
			\alpha=\frac{t(\mu_{2}(iy)-\mu_2(kiy))}{y} \qquad\text{and} \qquad \beta=y\left(\frac{t(\mu_{1}(iy)-\mu_1(kiy))}{(iy)^{2}}-M\right),
		\end{align}
		which satisfy
		\begin{align*}
			\alpha\ge \frac{t}{y}\min\left(\frac{1}{16}, \frac{(k-1)y}{8\pi}\right)>38 \quad\text{and}\quad |\beta|<\frac{2}{25} 
		\end{align*}
		by Lemma~4.2 of \cite{tyler} and the assumption of (\ref{eq-assum}). Substituting $\frac{1}{2\pi i}\log\frac{g_t(z,k)}{g_t(iy,k)}$ from \eqref{eq-tayfinal} in the integrand of \eqref{eq-p(N,t)}, we obtain
		\begin{align*}
			\notag
			\exp\left(-2\pi i Mx+2\pi i\frac{1}{2\pi i}\log\frac{g_{t}(z,k)}{g_{t}(iy,k)}\right)&=\exp\left(\frac{2\pi ix}{y}y\left(\frac{t(\mu_{1}(iy)-\mu_{1}(kiy))}{(iy)^{2}}-M\right)\right)\\
			\notag
			&\times\exp\left(-\pi\frac{x^{2}}{y^{2}}\frac{t(\mu_{2}(iy)-\mu_{2}(kiy))}{y}\left(1+2i\xi\frac{x}{y}+3\vartheta \frac{x^{2}}{y^{2}}\right)\right)\\
			&=\exp\left(\frac{2\pi i\beta x}{y}\right)\exp\left(-\pi\alpha\frac{x^{2}}{y^{2}}\left(1+2i\xi\frac{x}{y}+3\vartheta \frac{x^{2}}{y^{2}}\right)\right).
		\end{align*}
		By changing the variable $w=x/y$ and applying Lemma~4.1 of \cite{tyler}, we derive
		\begin{align*}
			\notag
			&\int_{-y/3}^{y/3}\exp\left(-2\pi i Mx+2\pi i\frac{1}{2\pi i}\log\frac{g_{t}(z,k)}{g_{t}(iy,k)}\right)dx\\&=\int_{-1/3}^{1/3}\exp\left(2\pi i\beta w\right)\exp\left(-\pi\alpha w^{2}\left(1+2i\xi w+\vartheta 3w^{2}\right)\right)y\,dw=\frac{y}{\sqrt{\alpha}}\left(1+\vartheta\frac{3.45}{\alpha}\right).
		\end{align*}
		It follows from the above equation and \eqref{eq-p(N,t)} that
		\begin{align}\label{eq-pNT}
			c_t(N,k) &=\exp(2\pi My)g_{t}(iy,k)\frac{y}{\sqrt{\alpha}}\left(1+\vartheta\frac{3.45}{\alpha}+\frac{\sqrt{\alpha}}{y}\vartheta y^{t} \exp\left(-\frac{t}{70}\min \left(k,\frac{1}{y}\right) \right)\right).
		\end{align}  
		Applying the upper bound from Lemma~4.2 of \cite{tyler} under the
		assumption that $\min (k,\frac{1}{y})\ge 1000$, we obtain
		\begin{align*}
			\frac{\alpha^{\frac{3}{2}}}{y} y^{t} \exp\left(-\frac{t}{70}\min \left(k,\frac{1}{y}\right) \right)<0.05.  
		\end{align*}
		Substituting $\alpha$ from \eqref{eq-alpha} in (\ref{eq-pNT}) yields
		\begin{align*}
			c_t(N,k)=\frac{y^{\frac{3}{2}}\exp(2\pi My)g_{t}(iy,k)}{\sqrt{t(\mu_{2}(iy)-\mu_2(kiy))}}\left(1+\vartheta\frac{3.5y}{t(\mu_{2}(iy)-\mu_2(kiy))}\right).    
		\end{align*}
		This completes the proof.
	\end{proof}
	We now turn to the proofs of our main results. We first recall $g_t(z,k)$ and $M$ from \eqref{eq-atz}, a key function used throughout this section:
	\begin{align*}
		g_{t}(z,k)=\frac{\eta(kz)^{kt}}{\eta(z)^{t}}\quad\text{and}\quad M=N+\frac{(k^{2}-1)t}{24}.
	\end{align*}
	In part $(i)$ of the following proof, we prove the uniqueness of the saddle point $y$, while in part $(ii)$, we establish a general asymptotic formula for $c_t(N,k)$.
	\begin{proof}[\textbf{\boldmath Proof of Theorem \ref{thm-1.1}$(i)$}]
		To establish the uniqueness of $y$, we first solve the equation
		$\frac{d}{dz}\left(-2\pi iMz+\log g_t(z,k)\right)=0$ at $z=iy.$ Hence, 
		\begin{align*}
			t\left(k\frac{d}{dz}\log \eta(kz)-\frac{d}{dz}\log \eta(z)\right)=2\pi iM.  
		\end{align*}
		Recall that $\frac{d}{dz}\log \eta(kz) = -\frac{2\pi i}{k z^2}\mu_1(kz)$, where $\mu_1$ is defined in \eqref{eq-muk}. Setting $z=iy$ yields
		\begin{equation*}
			\frac{t(\mu_{1}(kiy)-\mu_1(iy))}{y^{2}}=M=N+\frac{(k^{2}-1)t}{24}.   
		\end{equation*}
		Next, we prove that the solution $y>0$ is unique. From the explicit 
		expressions for $\mu$ given in \cite{tyler} (see equation $(4.14)$), we see that for $y\ge 1$,
		\begin{equation}\label{eq-pf11}
			\mu_{1}(iy)=\frac{y^{2}}{24}-\sum_{n=1}^{\infty}y^{2}\sigma(n)\exp(-2\pi ny). 
		\end{equation} 
		On the other hand, for $y < 1$, it follows that
		\begin{equation}\label{eq-pf12}
			\mu_{1}(iy)=-\frac{1}{24}+\frac{y}{4\pi}+\sum_{n=1}^{\infty}\sigma(n)\exp\left(-\frac{2\pi n}{y}\right).
		\end{equation}
		Combining the above two expressions, we see that
		\begin{align*}
			\lim_{y\to \infty} \frac{t(\mu_{1}(kiy)-\mu_1(iy))}{y^{2}} =\frac{(k^{2}-1)t}{24} \qquad \text{and}\qquad\lim_{y\to 0^{+}}\frac{t(\mu_{1}(kiy)-\mu_1(iy))}{y^{2}}=\infty.
		\end{align*}
		Hence, for any positive integer $N$, there exists some $y > 0$ such that
		\begin{equation*}
			\frac{t(\mu_{1}(kiy)-\mu_1(iy))}{y^{2}}=M=N+\frac{(k^{2}-1)t}{24}.   
		\end{equation*}
		We now prove that the solution $y>0$ is unique. From the explicit 
		formulas for $\mu_k$ given in \cite{tyler} (see (4.14) and (4.16)), 
		we obtain $$\mu_{2}(z)=-2\mu_{1}(z)+z\mu_{1}^{\prime}(z).$$ 
		Invoking Lemma~4.2 of \cite{tyler}, we arrive at
		\begin{align*}
			\frac{d}{dy}\left(\frac{t(\mu_{1}(kiy)-\mu_1(iy))}{y^{2}}\right)=\frac{t(\mu_{2}(kiy)-\mu_2(iy))}{y^{3}}<0.   
		\end{align*}
		Hence, the solution $y>0$ is unique.
	\end{proof}
	We now prove the second part of the theorem, namely the asymptotic formula for $c_t(N,k)$.
	\begin{proof}[\textbf{\boldmath Proof of Theorem~\ref{thm-1.1}(ii)}]
		We proved in \eqref{eq-t1.1} that $y$ is a unique solution of $\frac{t(\mu_{1}(kiy)-\mu_1(iy))}{y^{2}}=M$, so it is obvious that $\left|\frac{t(\mu_{1}(kiy)-\mu_1(iy))}{y^{2}}-M\right|\ll\frac{1}{y}$. Hence, from Proposition~\ref{thm-main}, we obtain
		\begin{align*}
			c_t(N,k)=\frac{y^{\frac{3}{2}}\exp(2\pi My)g_{t}(iy,k)}{\sqrt{t(\mu_{2}(iy)-\mu_2(kiy))}}\left(1+O\left(\frac{y}{t(\mu_{2}(iy)-\mu_2(kiy))}\right)\right).    
		\end{align*}
	\end{proof}
	We now establish an asymptotic formula for $c_t(N,k)$ when $ky\ge 1$.
	\begin{proof}[\textbf{\boldmath Proof of Theorem~\ref{thm-ctNm}}]
		By the proof of Lemma~3.3 of \cite{barman1}, for any $ky\ge 1$ and 
		\begin{align}\label{eq-delta}
			1<\delta<\frac{1}{(1-\exp(-2\pi ky))^{2}}<1.01,
		\end{align}
		we have
		\begin{align*}
			\eta(kiy)=\exp\left(-\frac{\pi ky}{12}-\delta\exp\left(-2\pi ky\right)\right). 
		\end{align*}
		For $y < 1$ and $1 < \nu < 1.01$, Lemma~3.4 of \cite{barman1} yields
		\begin{equation*}
			\eta(iy) = y^{-\frac{1}{2}} \exp\left( -\frac{\pi}{12y} - \nu e^{-\frac{2\pi}{y}} \right).
		\end{equation*}
		Applying the above estimates, we derive
		\begin{align*}
			g_t(iy,k)=y^{\frac{t}{2}}\exp\left(-\frac{\pi k^{2}ty}{12}-\delta kt\exp(-2\pi ky)+\frac{\pi t}{12y}+t\nu\exp\left(-\frac{2\pi }{y}\right)\right).   
		\end{align*}
		Inserting the above $g_t(iy,k)$ and Lemma~4.2 of \cite{tyler}$(iv)$ in Theorem~\ref{thm-1.1}$(ii)$ yields
		\begin{align}\label{eq-ctN2}
			c_t(N,k)= \frac{y^{\frac{t+3}{2}}\exp\left(2\pi Ny+\frac{\pi t}{12y}-\delta kt\exp(-2\pi ky)-\frac{\pi ty}{12}+t\nu\exp\left(-\frac{2\pi }{y}\right)\right)}{\sqrt{t(\mu_2(iy)-\mu_2(kiy))}}\left(1+O\left(y\right)\right). 
		\end{align}
		Now we solve $y$ explicitly from Theorem~\ref{thm-1.1}$(i)$. Combining (\ref{eq-pf11}) and (\ref{eq-pf12}), we deduce that
		\begin{align}\label{eq-pf23}
			y^{2}\left(N-\frac{t}{24}\right)+
			\frac{ty}{4\pi}-\frac{t}{24}\left(1-24S_1-24S_2\right)=0,
		\end{align}
		where
		\begin{align*}
			& S_1=\sum_{n=1}^{\infty}(ky)^{2}\sigma(n)\exp(-2\pi nky)\quad\text{and}\quad  S_2=\sum_{n=1}^{\infty}\sigma(n)\exp\left(-\frac{2\pi n}{y}\right).
		\end{align*}  
		Treating \eqref{eq-pf23} as a quadratic equation in $y$ and solving gives
		\begin{align*}
			y=-\frac{t}{8\pi\left(N-\frac{t}{24}\right)} +\frac{\sqrt{t}}{\sqrt{24\left(N-\frac{t}{24}\right)}}\sqrt{1+\frac{3t}{8\pi^{2}\left(N-\frac{t}{24}\right)}-24(S_1+S_2)}.   
		\end{align*} 
		Using a crude approximation $y\approx\sqrt{\frac{t}{24N}}$ in $S_{1}$ \footnote{Note that  $S_1=\sum_{n=1}^{\infty}(ky)^{2}\sigma(n)\exp(-2\pi nky)\le \frac{(\log N)^{2}}{16\pi^{2}\sqrt{N}}+O\left(\frac{(\log N)^{2}}{N}\right)$ for the specified range of $k$. This estimate is used in the derivation of (\ref{eq-soly}).} and $S_{2}$ \footnote{We also note that $S_2=\sum_{n=1}^{\infty}\sigma(n)\exp\left(-\frac{2\pi n}{y}\right)=O(N^{-2}).$}, and simplifying further, we obtain 
		\begin{equation}\label{eq-soly}
			y=\sqrt{\frac{t}{24N}}+\frac{C_{1}(N,k)}{N}+ \frac{C_{2}}{N}+O\left(\frac{(\log N)^{2}}{N^{\frac{3}{2}}}\right) . 
		\end{equation}
		is the solution of (\ref{eq-pf23}),  where $0<|C_1(N,k)|\le C_3(\log N)^{2}$ and, $C_2$ and $C_3$ are constants. We may also compute
		\begin{align}\label{eq-soly-inverse}
			y^{-1}=\sqrt{\frac{24N}{t}}\left(1-\frac{C_1(N,k)\sqrt{24}}{\sqrt{Nt}}-\frac{C_2\sqrt{24}}{\sqrt{Nt}}+O\left(\frac{(\log N)^{2}}{N}\right)\right).   
		\end{align}
		Combining \eqref{eq-soly} and \eqref{eq-soly-inverse} yields the following estimates:
		\begin{align}\label{eq-yy-inverse}
			\notag
			&2\pi Ny+\frac{\pi t}{12y}=\frac{2\pi}{\sqrt{6}}\sqrt{Nt}+O\left(\frac{(\log N)^{2}}{N^{\frac{1}{2}}}\right)\quad\text{and}\\
			&\exp\left(-\frac{\pi ty}{12}+tve^{-\frac{2\pi}{y}}\right)=1+O\left(N^{-\frac{1}{2}}\right).   
		\end{align}
		\begin{claim}\label{claim-2}
			Let $k\ge\frac{\sqrt{6}}{2\pi}\sqrt{\frac{N}{t}}\log \frac{N}{t}$. Then,
			\begin{equation*}
				\mu_{2}(iy)-\mu_{2}(kiy)=\frac{1}{12}\left(1+O\left(N^{-\frac{1}{2}}(\log N)^{3}\right)\right).
			\end{equation*} 
		\end{claim}
		\begin{proof}[\textbf{\boldmath Proof of Claim \ref{claim-2}}]
			For given $y$, equation~(4.14) of \cite{tyler} gives
			\begin{align*}
				\mu_{2}(iy)&=\sum_{n=1}^{\infty}\left(\frac{2\pi n}{y}-2\right)\sigma(n) \exp\left(-\frac{2\pi n}{y}\right)+\frac{1}{12}-\frac{y}{4\pi}=\frac{1}{12}\left(1+O\left(N^{-\frac{1}{2}}\right)\right).
			\end{align*}
			Since $k\ge\frac{\sqrt{6}}{2\pi}\sqrt{\frac{N}{t}}\log \frac{N}{t}$ and $y$ is given by \eqref{eq-soly}, it follows from equation~(4.14) of \cite{tyler} that 
			\begin{equation*}
				\mu_{2}(kiy) = \sum_{n=1}^{\infty}(ky)^{3}(2\pi n)\sigma(n)\exp(-2\pi nky)=O\left(N^{-\frac{1}{2}}(\log N)^{3}\right)
			\end{equation*}
			Hence,
			\begin{equation*}
				\mu_{2}(iy)-\mu_{2}(kiy)=\frac{1}{12}\left(1+O\left(N^{-\frac{1}{2}}(\log N)^{3}\right)\right).
			\end{equation*}
		\end{proof}
		Substituting $y = \sqrt{\frac{t}{{24N}}}\left(1+O\left(N^{-\frac{1}{2}}\right)\right)$ from \eqref{eq-soly}, together with the estimates from \eqref{eq-yy-inverse} and the bound for $\mu_{2}(iy)-\mu_{2}(kiy)$ from the preceding claim in \eqref{eq-ctN2}, we simplify the expression for $c_t(N,k)$ as follows:
		\begin{align*}
			c_t(N,k)=p_t(N)\exp\left(-\delta kt\exp\left(-\frac{\pi k\sqrt{t}}{\sqrt{6N}}\right)\right)\left(1+O\left(N^{-\frac{1}{2}}(\log N)^{3}\right)\right),
		\end{align*}
		by applying the explicit formula for $p_t(N)$ given in Theorem~\ref{thm-murty}. For the given range of $k$ and $y$, we have $\delta=1+O(N^{-\frac{1}{2}})$, where $\delta$ is defined in \eqref{eq-delta}.
	\end{proof}
	\section{Proof of Theorem~\ref{thm-ztNm}}
	In this section, we derive a lower bound for $Z_t(N)$. To this end, we first prove two auxiliary lemmas that will play a key role in the proof of the main result.
	\begin{lemma}\label{lemma-erdos}
		The number of $t$-multipartitions $\mu$ with two parts $k_1$ and $k_2$ both larger than $K$ is
		\begin{align*}
			O\left(\frac{t^{2}p_t(N)}{\left(\log \frac{N}{t}\right)^{2}}\right),  
		\end{align*}
		where $C=\frac{2\pi }{\sqrt{6}}$ and 
		\begin{align*}
			K= C^{-1}\sqrt{\frac{N}{t}}\left(1+\frac{1}{B}\right)\log \frac{N}{t}-2C^{-1}\sqrt{\frac{N}{t}}\log \log \log \frac{N}{t} \quad\text{with} \quad \left(\frac{N}{t}\right)^{\frac{1}{2B}}=C^{-1}\log \frac{N}{t}.
		\end{align*}
	\end{lemma}
	\begin{proof}
		The number of $t$-multipartitions $\mu$ having two parts $k_1$ and $k_2$, both exceeding $K$, is bounded by
		\begin{align*}
			\sum_{k_2,k_1>K}t^{2}p_t(N-k_1-k_2).
		\end{align*}
		The parts $k_1$ and $k_2$ can be chosen from any of the $t$ components of $\mu$, giving $t^2$ choices. Once these parts are fixed, the remaining parts form a $t$-multipartition of $N-k_1-k_2$, which can be chosen in $p_t(N-k_1-k_2)$ ways.
		
		Using the asymptotic estimate for $p_t(N)$ given in Theorem~\ref{thm-murty}, we deduce
		\begin{align*}
			&\ll t^{2}p_t(N)\sum_{k_2,k_1>K}\left(\frac{N}{N-k_1-k_2}\right)^{\frac{t+3}{4}}\exp\left(\frac{2\pi}{\sqrt{6}}\left(\sqrt{t(N-k_1-k_2)}-\sqrt{Nt}\right)\right)\\
			&\ll t^{2}p_t(N)\sum_{k_2,k_1>K}\exp\left(-\frac{2\pi}{\sqrt{6}}\sqrt{\frac{t}{N}}\frac{k_1+k_2}{2}\right)\\
			&\ll t^{2}p_t(N)\left(\sum_{k>K}\exp\left(-\frac{2\pi}{\sqrt{6}}\sqrt{\frac{t}{N}}\frac{k}{2}\right)\right)^{2}\\
			&=t^{2}p_t(N)\left(\frac{\exp\left(-\frac{2\pi}{\sqrt{6}}\sqrt{\frac{t}{N}}\frac{K}{2}\right)}{1-\exp\left(-\frac{2\pi}{\sqrt{6}}\sqrt{\frac{t}{N}}\frac{1}{2}\right)}\right)^{2}\\
			&\ll \frac{t^{2}p_t(N)}{\left(\log \frac{N}{t}\right)^{2}}.
		\end{align*}   
	\end{proof}
	\begin{remark}
		The construction of $K$ in the above lemma follows the approach used in the proof of Theorem~1.3 of \cite{barman1}.
	\end{remark}
	\begin{lemma}\label{lemma-erdos0}
		Let $N$ be a large positive integer and $1\le k\le N$, then 
		\begin{align*}
			p_t(N-k)=p_t(N)&\exp\left(-\frac{\pi k}{\sqrt{6}}\sqrt{\frac{t}{N}}\right)\left(1+O\left(\min\left(1,\frac{k}{N^{\frac{3}{4}}}\right)\right)\right). 
		\end{align*}
	\end{lemma}
	\begin{proof}
		Applying Theorem~\ref{thm-murty}, we obtain
		\begin{align*}
			\frac{p_t(N-k)}{p_t(N)}&=\frac{N^{\frac{t+3}{4}}}{(N-k)^{\frac{t+3}{4}}}\exp\left(\frac{2\pi}{\sqrt{6}}\left(\sqrt{(N-k)t}-\sqrt{Nt}\right)\right)\left(1+O\left(N^{-\frac{1}{2}}\right)\right)  \\
			&=\frac{1}{\left(1-\frac{k}{N}\right)^{\frac{t+3}{4}}}\exp\left(-\frac{\pi k}{\sqrt{6}}\sqrt{\frac{t}{N}}-\frac{\pi k^{2}}{4\sqrt{6}}\frac{\sqrt{t}}{N^{\frac{3}{2}}}-\cdots\right)\left(1+O\left(N^{-\frac{1}{2}}\right)\right)\\
			&=\exp\left(-\frac{\pi k}{\sqrt{6}}\sqrt{\frac{t}{N}}\right)\left(1+O\left(\min\left(1,\frac{k}{N^{\frac{3}{4}}}\right)\right)\right).
		\end{align*}
	\end{proof}
	We are now ready to prove the desired lower bound for $Z_t(N)$.
	\begin{proof}[\textbf{\boldmath Proof of Theorem~\ref{thm-ztNm}}]
		We first consider the inequality given in \eqref{eq-main}:
		\begin{align}\label{eq-main1}
			Z_t(N)&\ge \sum_{k=1}^{N}c_t(N,k)\left(tp_{t,k}(N-k)-t^{2}p_{t,k}(N-2k)\right)=S_1+S_2,
		\end{align}
		where 
		\begin{align*}
			&S_1 =\sum_{k=1}^{K}c_t(N,k)\left(tp_{t,k}(N-k)-t^{2}p_{t,k}(N-2k)\right),\\
			&S_2=\sum_{k>K}^{N}c_t(N,k)\left(tp_{t,k}(N-k)-t^{2}p_{t,k}(N-2k)\right),
		\end{align*}
		and $C=\frac{2\pi }{\sqrt{6}}$, and
		\begin{align*}
			K= C^{-1}\sqrt{\frac{N}{t}}\left(1+\frac{1}{B}\right)\log \frac{N}{t}-2C^{-1}\sqrt{\frac{N}{t}}\log \log \log \frac{N}{t} \quad\text{with} \quad \left(\frac{N}{t}\right)^{\frac{1}{2B}}=C^{-1}\log \frac{N}{t}.
		\end{align*}
		\begin{claim}\label{claim-S1}
			\begin{align*}
				S_1=O\left(\frac{p_t(N)^{2}\log \log \frac{N}{t}}{\left(\log \frac{N}{t}\right)^{2}}\right).  
			\end{align*}    
		\end{claim}
		\begin{proof}[\textbf{\boldmath Proof of Claim~\ref{claim-S1}}]
			Let $K_0= C^{-1}\sqrt{\frac{N}{t}}\left(1+\frac{1}{2B}\right)\log \frac{N}{t}$. We divide the sum in $S_1$ into two parts $\sum_{k=1}^{K_0}$, $\sum_{k=K_0+1}^{K}$. Since from Theorem~\ref{thm-ctNm}, we can show that $c_t(N,K_0)=O\left(\frac{p_t(N)^{2}}{ \left(\log \frac{N}{t}\right)^{2}}\right)$, the first part becomes
			\begin{align*}
				\sum_{k=1}^{K_0} c_t(N,k)\left(tp_{t,k}(N-k)-t^{2}p_{t,k}(N-2k)\right)\ll c_t(N,K_0)p_t(N)=O\left(\frac{p_t(N)^{2}}{ \left(\log \frac{N}{t}\right)^{2}}\right). 
			\end{align*}
			Next, for the second sum, we have
			\begin{align*}
				\sum_{k=K_0+1}^{K} c_t(N,k)\left(tp_{t,k}(N-k)-t^{2}p_{t,k}(N-2k)\right)&\ll c_t(N,K)\sum_{k=K_0+1}^{K} p_t(N-k))\\
				&\ll \frac{p_t(N)^{2}}{ \left(\log \frac{N}{t}\right)^{2}}.
			\end{align*}
			The desired bound then follows from Theorem~\ref{thm-ctNm} and Lemma~\ref{lemma-erdos0}.
		\end{proof}
		\begin{claim}\label{claim-3}
			\begin{align*}
				S_2=\frac{2p_t(N)^{2}}{\log \frac{N}{t}}\left(1+O\left(\frac{\log \log \frac{N}{t}}{\log \frac{N}{t}}\right)\right).
			\end{align*} 
		\end{claim}
		\begin{proof}[\textbf{\boldmath Proof of Claim~\ref{claim-3}}]		
			By Lemma~\ref{lemma-erdos}, $t$-multipartitions $\mu$ containing two or more parts larger than $K$ contribute a negligible error. Consequently, we may drop the uniqueness condition and consider $t$-multipartitions containing one part $k > K$, yielding
			\begin{align*}
				\sum_{k>K} p_{t,k}(N-k)=\sum_{k>K}p_t(N-k)+O\left(\frac{p_t(N)}{\left(\log \frac{N}{t}\right)^{2}}\right).      
			\end{align*} 
			Since, in the sum $S_2$, we have already taken into account the terms $p_t(N-2k)$ with $k>K$, it follows from Lemma~\ref{lemma-erdos} that
			\begin{align*}
				\sum_{k>K} t^{2}p_{t,k}(N-2k)=O\left(\frac{p_t(N)}{\left(\log \frac{N}{t}\right)^{2}}\right).
			\end{align*} 
			Combining the above two estimates and using the fact that $c_t(N,k)\le p_t(N)$, we obtain
			\begin{align}\label{eq-ctNm}
				S_2&=t\sum_{k>K}c_t(N,k)p_t(N-k)+O\left(\frac{p_t(N)^{2}}{\left(\log \frac{N}{t}\right)^{2}}\right).
			\end{align}
			Inserting Theorem~\ref{thm-ctNm} and Lemma~\ref{lemma-erdos0}, we derive
			\begin{align}\label{eq-S2bar}
				\notag
				\bar{S_2}&=t\sum_{k>K}c_t(N,k)p_t(N-k)\\
				\notag
				&= tp_t(N)^{2}\sum_{k>K}u^{k}\exp\left(-ktu^{k}\right)\left(1+O\left(\min\left(1,\frac{k}{N^{\frac{3}{4}}}\right)\right)\right)\left(1+O\left(\frac{1}{\log \frac{N}{t}}\right)\right)\\
				&=tp_t(N)^{2}\sum_{k>K}f(k)\left(1+O\left(\min\left(1,\frac{k}{N^{\frac{3}{4}}}\right)\right)\right)\left(1+O\left(\frac{1}{\log \frac{N}{t}}\right)\right),
			\end{align}
			where $u=\exp\left(-\frac{\pi }{\sqrt{6}}\sqrt{\frac{t}{N}}\right)$ and $f(k)=u^{k}\exp\left(-ktu^{k}\right)$.\\
			Since  
			\begin{align*}
				\frac{d}{dv}f(v)&=u^{v}\left((\log u)\exp\left(-vtu^{v}\right)-(tu^{v}+vtu^{v}\log u)\exp\left(-vtu^{v}\right)\right) \\
				&=O\left(u^{k}N^{-\frac{1}{2}}\right) \quad \text{(for $k\le v\le k+1$ and $k>K$).}
			\end{align*}
			Applying the first-order Euler--Maclaurin summation formula over the interval $[k, k+1]$ and for $k>K$, we deduce
			\begin{align*}
				f(k) &= \int_{k}^{k+1}f(v) dv + O\left( \max_{k \le v \le k+1} \left| \frac{d}{dv} f(v) \right| \right)\\
				&=\int_{k}^{k+1}f(v) dv + O\left(u^{k}N^{-\frac{1}{2}}\right).
			\end{align*}
			Hence,
			\begin{align}\label{eq-I1}
				\notag
				\sum_{k>K}f(k)&=\int_{K}^{\infty}f(v)dv+O\left(\sum_{k>K}u^{k}N^{-\frac{1}{2}}\right)\\
				\notag
				&=\int_{K}^{\infty}f(v)dv+O\left(N^{-\frac{1}{2}}\frac{u^{K}}{u^{-1}-1}\right) \\
				&=\int_{K}^{\infty}f(v)dv+O\left(N^{-\frac{1}{2}}\right). 
			\end{align}
			Setting $I_1 = \int_{K}^{\infty} f(v) \, dv$, and applying the change of variable $w=v-K$ yields
			\begin{align*}
				I_1&=\int_{0}^{\infty}u^{w+K}\exp\left(-(w+K)tu^{w+K}\right)dw \\
				&=\int_{0}^{\infty}u^{w+K}\exp\left(-Ktu^{w+K}\right)\exp\left(-wtu^{w+K}\right)dw. 
			\end{align*}
			Since $u=\exp\left(-\frac{\pi }{\sqrt{6}}\sqrt{\frac{t}{N}}\right)$ and $w>0$, we may write
			\begin{align*}
				\exp\left(-wtu^{w+K}\right)=1+O\left(wtu^{w+K}\right)=1+O\left(wu^{w}tu^{K}\right).
			\end{align*}
			The inequality $wu^{w}\le \frac{1}{\log \frac{1}{u}}$ hold because $wu^{w}$ has a global maximum at $w=\frac{1}{\log \frac{1}{u}}$. Hence,
			\begin{align*}
				\exp\left(-wtu^{w+K}\right)&= 1+O\left(\frac{tu^{K}}{\log \frac{1}{u}}\right)= 1+O\left(\frac{1}{\log \frac{N}{t}}\right).  
			\end{align*}
			Therefore,
			\begin{align*}
				I_1&=\int_{0}^{\infty}u^{w+K}\exp\left(-Ktu^{w+K}\right)dw\left(1+O\left(\frac{1}{\log \frac{N}{t}}\right)\right).    
			\end{align*}
			Using the change of variables $y = u^{w+K}$, we obtain
			\begin{align*}
				I_1&=\frac{1}{\log u}\int_{u^{K}}^{0}\exp\left(-Kty\right)dy\left(1+O\left(\frac{1}{\log \frac{N}{t}}\right)\right) \\
				&=\frac{1-\exp\left(-Ku^{K}\right)}{Kt\log\left(\frac{1}{u}\right)}=\frac{2}{t\log \frac{N}{t}}\left(1+O\left(\frac{\log \log \frac{N}{t}}{\log \frac{N}{t}}\right)\right).
			\end{align*}
			Substituting the above estimate for $I_1$ in \eqref{eq-I1}, then using \eqref{eq-I1} in \eqref{eq-S2bar}, and finally applying \eqref{eq-ctNm}, we obtain
			\begin{align*}
				S_2= \frac{2p_t(N)^{2}}{\log \frac{N}{t}}\left(1+O\left(\frac{\log \log \frac{N}{t}}{\log \frac{N}{t}}\right)\right).   
			\end{align*}
		\end{proof}
		Substituting the results of Claim~\ref{claim-S1} and Claim~\ref{claim-3} in \eqref{eq-main1}, we obtain
		\begin{align*}
			Z_t(N)\ge \frac{2p_t(N)^{2}}{\log \frac{N}{t}}\left(1+O\left(\frac{\log \log \frac{N}{t}}{\log \frac{N}{t}}\right)\right).   
		\end{align*}
		This completes the proof.
		
	\end{proof}
	\subsection*{AI Declaration}
	
	The authors declare that no artificial intelligence tools or AI-assisted technologies were
	used in the preparation, writing, or mathematical analysis of this manuscript.

\end{document}